\documentclass[reqno, 11pt, a4paper, oneside]{amsart}
\usepackage{amsmath}
\usepackage{amsfonts}
\usepackage{amssymb}
\usepackage[ansinew]{inputenc}
\usepackage{graphicx}
\usepackage{amsbsy}
\usepackage{fancyhdr}
\usepackage{yfonts}
\usepackage{hyperref}
\usepackage{mathabx}
\usepackage{enumerate}

\numberwithin{equation}{section}
\newtheorem{teo}{Theorem}[section]
\newtheorem{prop}[teo]{Proposition}
\newtheorem{lema}[teo]{Lemma}

\newtheorem{coro}[teo]{Corollary}
\theoremstyle{definition}
\newtheorem{defi}[teo]{Definition}
\newtheorem{notation}[teo]{Notation}

\theoremstyle{remark}

\sffamily
\title{Phase relations and pyramids}
\author{Miguel N. Walsh}
\address{Departamento de Matemática e IMAS-CONICET, Facultad de Ciencias Exactas y Naturales, Universidad de Buenos Aires, 1428 Buenos Aires, Argentina}
\email{mwalsh@dm.uba.ar}

\begin{document}

\def\F{\mathbb{F}}
\def\Fqn{\mathbb{F}_q^n}
\def\Fq{\mathbb{F}_q}
\def\Fp{\mathbb{F}_p}
\def\Di{\mathbb{D}}
\def\E{\mathbb{E}}
\def\Z{\mathbb{Z}}
\def\Q{\mathbb{Q}}
\def\C{\mathbb{C}}
\def\R{\mathbb{R}}
\def\N{\mathbb{N}}
\def\H{\mathbb{H}}
\def\P{\mathbb{P}}
\def\T{\mathbb{T}}
\def\modp{\, (\text{mod }p)}
\def\modN{\, (\text{mod }N)}
\def\modq{\, (\text{mod }q)}
\def\modone{\, (\text{mod }1)}
\def\ZN{\mathbb{Z}/N \mathbb{Z}}
\def\Zp{\mathbb{Z}/p \mathbb{Z}}
\def\Zan{a^{-n}\mathbb{Z}/ \mathbb{Z}}
\def\Zal{a^{-l} \Z / \Z}
\def\Pr{\text{Pr}}
\def\leftsize{\left| \left\{}
\def\rightsize{\right\} \right|}

\maketitle

\begin{abstract}
We develop tools to study the averaged Fourier uniformity conjecture and extend its known range of validity to intervals of length at least $\exp(C (\log X)^{1/2} (\log \log X)^{1/2})$.
\end{abstract}

\bigskip

\section{Introduction}

In this article we shall establish the following result.

\begin{teo}
\label{1}
Given $0 < \rho < 1$ and $\eta > 0$, there exists some $C>0$ such that, for every $\exp(C (\log X)^{1/2} (\log \log X)^{1/2} ) \le H \le X^{1/2}$ and every complex-valued multiplicative function $g$ with $|g| \le 1$ satisfying
\begin{equation}
\label{A}
 \int_X^{2X} \sup_{\alpha} \left| \sum_{x \le n \le x+H} g(n) e(\alpha n) \right| dx \ge \eta H X,
 \end{equation}
we have $\mathbb{D} (g; CX^2/H^{2-\rho},C) \le C$.
\end{teo}

Here, we are writing $\mathbb{D}(g;T,Q)$ for the 'pretentious' distance \cite{GS}:
$$ \mathbb{D}(g;T,Q) = \inf \left( \sum_{p \le T} \frac{1 - \text{Re}(g(p)p^{it}\chi(p))}{p} \right)^{1/2} ,$$
with the infimum taken over all $|t| \le T$ and all Dirichlet characters of modulus at most $Q$. In particular, Theorem \ref{1} implies that (\ref{A}) cannot hold for the Möbius and Liouville functions.

Theorem \ref{1} improves on estimates obtained in \cite{MRT} and \cite{MRTTZ}, where it was shown that for any $\varepsilon > 0$ the result holds for intervals of length at least $X^{\varepsilon}$ and $\exp((\log X)^{5/8+\varepsilon})$, respectively. The methods of this article should adapt to nilsequences, thus yielding corresponding progress on the higher uniformity conjecture. It is known that after passing to logarithmic averages, establishing this conjecture for intervals of length at least $(\log X)^{\varepsilon}$, for every $\varepsilon >0$, would imply both Chowla's and Sarnak's conjectures \cite{MRTTZ,T}.

We will proceed through the same general framework as in previous articles \cite{MRT, MRTTZ,W}, where it is shown that a function $g$ satisfying (\ref{A}), for the corresponding values of $H$, must correlate with $n \mapsto e(an/q) n^{2 \pi i T}$ on many of these intervals and for certain fixed choices of $T$ and $q$. The results of \cite{MR,MR2, MRT2} then imply that $g$ must behave globally like a function of this form, yielding the desired conclusion.

To obtain this local correlation one exploits Elliott's inequality and the large sieve (see \cite{MRT, MRTTZ}), which combined guarantee that under (\ref{A}) we may associate to many intervals $I \subseteq [X,2X]$ a frequency $\alpha_I \in \R$ in such a way that we may find many quadruples consisting of a pair of intervals $I=[x,x+H], J= [y,y+H]$ and a pair of primes $p,q \le H$, such that $|x/p-y/q|$ is small and $p \alpha_I$ close to $q \alpha_J$ mod $Q$, for some large integer $Q$. The problem then becomes that of showing that these relations force $\alpha_I$ to be close mod $1$ to $\frac{m_I}{q} + \frac{T}{x}$ for certain $m_I,q \in \N$ and $T \in \R$, which would then imply the desired local correlation.

To accomplish this we begin in Section \ref{phase} by studying 'pre-paths' consisting of a sequence of elements $\alpha_1, \ldots, \alpha_{k+1} \in \Z / Q \Z$ and primes $p_1, \ldots, p_k, q_1, \ldots, q_k$ with $p_i \alpha_i$ close to $q_i \alpha_{i+1}$, for every $1 \le i \le k$. Building on ideas of \cite{W}, we show how to construct certain 'pyramids' of frequencies that help relate the elements of the sequence and, in particular, a 'top' element $\alpha$ such that $\alpha_i$ is close to $\left( \prod_{j=1}^{i-1} p_j \prod_{j=i}^k q_j \right) \alpha$, for every $1 \le i \le k+1$, with an error that depends on the relative sizes of the primes involved.

After developing our general setting further in Section 3, we proceed in Section 4 to show how the additional 'physical' information that $|x/p - y/q|$ is small can be used to obtain uniform bounds for the approximations of the previous paragraph. In particular, the conclusions attained would work equally well for $H$ in the poly-logarithmic range and seem likely to be useful tools for future work on these problems.

The purpose of Section 5 is then to show that for many of the intervals we are studying the corresponding frequencies are connected by paths of the above form. Some connectedness of this type seems necessary in order to be able to find a fixed choice of $q$ and $T$ that works for many of these intervals and here is where the required lower bound on $H$ in Theorem \ref{1} emerges. The reason for it is that when studying paths of length $k$, one is naturally led to some losses of the order of $C^k$ in the bounds, for some absolute constant $C>1$. Such losses become problematic once $C^{k}$ is comparable to $H$ and since one needs to consider paths of length around $\frac{\log X}{\log H}$ in order to be able to have enough of the intervals connected with each other, we end up in that situation once $H$ goes below $\exp( C (\log X)^{1/2})$ (essentially the same observation can already be found in \cite{MRT, MRTTZ}). In fact, due to the density of prime numbers one is also led to some additional factors of the order of $(\log H)^k$, which is the reason for the exact lower bound on $H$ in Theorem \ref{1}. 

Finally, we complete the proof in Section 6. Once enough intervals have been connected to a fixed interval $I_0$, one can relatively easy use the properties of the 'pyramids' obtained in the first sections to show that a fixed choice of $q$ and $T$ works for many of the intervals.

We notice that the methods of this article end up using auxiliary phases as in \cite{W}, but also the advantage of working with higher moduli as in \cite{MRT, MRTTZ}. As such, they can be seen as a middle point between both arguments. On the other hand, an outcome of this article is that neither the contagion arguments of \cite{W} nor the mixing lemmas of \cite{MRT, MRTTZ} end up being necessary to cover the natural range of $\exp( (\log X)^{1/2+\varepsilon})$. However, such tools may very well end up being useful when trying to lower the value of $H$ further.

\begin{notation} 
We will write $X \lesssim Y$ or $X=O(Y)$ to mean that there is some absolute constant $C$ with $|X| \le C Y$ and $X \sim Y$ if both $X \lesssim Y$ and $Y \lesssim X$ hold. If the implicit constants depend on some additional parameters, we shall use a subscript to indicate this. For a finite set $S$ we write $|S|$ for its cardinality. We abbreviate $e(x):=e^{2 \pi i x}$ and given $Q \in \N$ we write $\| \cdot \|_Q$ for the distance to $0$ in $\R / Q \Z$.
\end{notation}

\section{Pyramids}
\label{phase}

We begin with the following extension of \cite[Lemma 2.1]{W}.

\begin{lema}
\label{astart}
Let $\epsilon_1,\epsilon_2 > 0$ and let $Q \in \N$. Let $\alpha_1,\alpha_2 \in \R / Q \Z$ and let $p_1,p_2$ be distinct primes not dividing $Q$ with $\| p_1 \alpha_1 - p_2 \alpha_2\|_Q  < \epsilon_1 + \epsilon_2$. Then, there exists $\alpha \in \R / Q \Z$ with $\| p_i \alpha - \alpha_j \|_Q < \frac{\epsilon_j}{p_j}$ if $i \neq j$.
\end{lema}

\begin{proof}
For $\left\{ i, j \right\}= \left\{ 1,2 \right\}$, let $\alpha^{(i)} \in \R / Q \Z$ be such that $p_i \alpha^{(i)} = \alpha_j$. Adding integer multiples of $Q/p_i$ to $\alpha^{(i)}$ we may assume that $\| \alpha^{(1)} - \alpha^{(2)} \|_Q \le \frac{Q}{2 p_1 p_2}$. On the other hand, we have by hypothesis that $\| p_1 p_2 (\alpha^{(1)} - \alpha^{(2)} ) \|_Q < \epsilon_1 + \epsilon_2$. Combining both estimates we see that it must in fact be $\| \alpha^{(1)} - \alpha^{(2)} \|_Q < \frac{\epsilon_1 + \epsilon_2}{p_1 p_2}$. Taking $\alpha= \alpha^{(1)} - \frac{\epsilon_2}{\epsilon_1+\epsilon_2} (\alpha^{(1)}-\alpha^{(2)})$ we obtain the result.
\end{proof}

Lemma \ref{astart} immediately implies the following corollary.

\begin{lema}
\label{tcp}
Let $\epsilon_j,\epsilon_j'>0$ for every $1 \le j \le k$ and let $Q \in \N$. Let $(\alpha_1^{(1)},\ldots,\alpha^{(1)}_{k+1})$ and $(\alpha_2^{(0)},\ldots,\alpha_{k+1}^{(0)})$ be tuples of elements of $\R / Q \Z$ and let $p_1,\ldots,p_{k}$, $q_1,\ldots,q_{k}$ be a sequence of distinct primes not dividing $Q$ such that, for every $1 \le j \le k$, we have $\| q_{j} \alpha^{(1)}_{j+1} - \alpha^{(0)}_{j+1} \|_Q < \epsilon_j$ and $\| p_j \alpha^{(1)}_j - \alpha^{(0)}_{j+1} \|_Q < \epsilon'_j$. Then, there exists a tuple $(\alpha^{(2)}_1,\ldots,\alpha^{(2)}_{k})$ of elements of $\R /Q  \Z$ such that, for every $1 \le j \le k$, we have $\| p_j \alpha^{(2)}_j  - \alpha^{(1)}_{j+1} \|_Q < \frac{\epsilon_j}{q_j}$ and $\| q_j \alpha^{(2)}_j - \alpha^{(1)}_j \|_Q < \frac{\epsilon_j'}{p_j}$.
\end{lema}

This can be iterated in the following way: after we have used sets of frequencies $(\alpha_1^{(j)},\ldots,\alpha^{(j)}_{k+2-j})$, $(\alpha_2^{(j-1)},\ldots,\alpha_{k+2-j}^{(j-1)})$ and primes $p_1,\ldots,p_{k+1-j}$, $q_j,\ldots,q_{k}$ not dividing $Q$ to obtain a new set $(\alpha^{(j+1)}_1,\ldots,\alpha^{(j+1)}_{k+1-j})$, we can then use the sets of frequencies $(\alpha^{(j+1)}_1,\ldots,\alpha^{(j+1)}_{k+1-j})$, $(\alpha_2^{(j)},\ldots,\alpha^{(j)}_{k+1-j})$ and primes $p_1,\ldots,p_{k-j}$, $q_{j+1},\ldots,q_{k}$ to obtain a new set $(\alpha^{(j+2)}_1,\ldots,\alpha^{(j+2)}_{k-j})$. We thus arrive, for every $1 \le j \le k+1$, at a tuple $(\alpha_1^{(j)},\ldots,\alpha_{k+2-j}^{(j)})$ of elements of $\R / Q \Z$. It will be convenient to name these objects.

\begin{defi}
A \emph{pre-path} mod $Q$ is a choice of (ordered) tuples $(\alpha_1^{(1)},\ldots,\alpha^{(1)}_{k+1})$, $(\alpha_2^{(0)},\ldots,\alpha_{{k+1}}^{(0)})$ of elements of $\R / Q \Z$, real numbers $\epsilon_1, \ldots, \epsilon_{k}, \epsilon_1', \ldots, \epsilon'_{k}>0$ and distinct primes $p_1,\ldots,p_{k}$, $q_1,\ldots,q_{k}$ not dividing $Q$ satisfying the hypothesis of Lemma \ref{tcp}. We say a corresponding sequence $(\alpha_1^{(j)})_{1 \le j \le k+1}$ obtained as in the previous paragraph is a \emph{pyramid} associated with this pre-path. We call $\alpha_1^{(k+1)}$ the \emph{top element} of the pyramid and $k$ the \emph{length} of the pre-path.
\end{defi}

The following will be our main input for the study of pre-paths.

\begin{lema}
\label{premanli}
Let $0 < \epsilon < 1$. Given a pre-path of length $k$ with $\epsilon_i=\epsilon_i'=\epsilon$ for every $1 \le i \le k$, we have that any associated pyramid $(\alpha_1^{(j)})_{1 \le j \le k+1}$ satisfies
$$ \| q_j \alpha_1^{(j+1)} - \alpha_1^{(j)} \|_Q <  \epsilon \left( \prod_{i=1}^{\lfloor (j-1)/2 \rfloor+1} p_i \right)^{-1} \left( \prod_{i=1}^{\lceil (j-1)/2 \rceil } q_{j-i} \right)^{-1},$$
for every $1 \le j \le k$.
\end{lema}

\begin{proof}
We will proceed by induction on the length of the pre-path. If $k=1$ the claim is immediate from Lemma \ref{tcp}, so we let $k \ge 2$ and assume the result has already been established for all pre-paths of length at most $k-1$. If $j \le k-1$ observe that the parameters $(\alpha_1^{(1)},\ldots,\alpha_{j+1}^{(1)})$, $(\alpha_2^{(0)},\ldots,\alpha_{j+1}^{(0)})$, $p_1,\ldots,p_{j}$, $q_1,\ldots,q_{j}$ and $\epsilon_1=\epsilon'_1=\ldots=\epsilon_j=\epsilon_j'=\epsilon$ form a pre-path of length $j$ and $(\alpha_1^{(i)})_{1 \le i \le j+1}$ is a pyramid for this pre-path, so the result follows by induction in this case. We thus only need to treat $j=k$. By Lemma \ref{tcp} we know that 
\begin{equation}
\label{ppp14}
 \| q_{k} \alpha_1^{(k+1)} - \alpha_1^{(k)} \|_Q \le \frac{\| p_1 \alpha_1^{(k)} - \alpha_2^{(k-1)} \|_Q}{p_1}.
 \end{equation}
Here we are using that, after iterating Lemma \ref{tcp}, $q_{k}$ and $p_1$ are the primes that end up relating $\alpha_1^{(k+1)}$ with $\alpha_1^{(k)}$ and $\alpha_2^{(k)}$, respectively. Notice now that the 'inverted' parameters $(\alpha_{k}^{(1)},\ldots,\alpha_1^{(1)})$, $(\alpha_{k}^{(0)},\ldots,\alpha_2^{(0)})$, $q_{k-1},\ldots,q_1$, $p_{k-1},\ldots,p_1$ and $\epsilon$ form a pre-path of length $k-1$ and that $(\alpha_{k+1-j}^{(j)})_{1 \le j \le k}$ is a pyramid for this pre-path (notice that the roles of the primes $p_i$ and $q_i$ have also been inverted). It thus follows by induction that the right-hand side of (\ref{ppp14}) is
$$ <\frac{\epsilon}{p_1} \left( \prod_{i=1}^{\lfloor k/2 \rfloor} q_{k-i} \right)^{-1} \left( \prod_{i=1}^{\lceil k/2-1 \rceil } p_{(k)-(k-1-i)} \right)^{-1}.$$
The result then follows upon rearranging. 
\end{proof}

\section{General setting}

In this section we will invoke some results from \cite{MRT} and \cite{MRTTZ} that will serve as the setting for our approach and will also establish a regularity estimate that will be useful in the rest of the article.

As mentioned in the introduction, once we have shown that $g$ correlates with $n \mapsto e(an/q) n^{2 \pi i T}$ on many of the intervals and for a certain fixed choice of $q$ and $T$, Theorem \ref{1} will then follow from the results of Matömaki and Radziwi\l\l \, \cite{MR,MR2, MRT2}. More precisely, we will be using the following reduction which follows from the last part of \cite[Section 6]{MRTTZ} and relies on the power saving bounds of \cite{MR2}.

\begin{prop}
\label{2}
Let $g$ be a complex-valued multiplicative function with $|g| \le 1$. Let $\rho>0$ be sufficiently small, $C >0$ sufficiently large with respect to $\rho$ and $X \ge 1$ sufficiently large with respect to $\rho$ and $C$. Let $H$ be as in Theorem \ref{1}. Also, let $T \in \R$ and $q \in \N$ satisfy $|T| \le C X^2/H^{2-\rho}$ and $q \le C H^{\rho}$. Assume that for $\ge X/H^{1+\rho}$ disjoint intervals $I \subseteq [X,2X]$ of length $H^* \in [H^{1-\rho}, H]$ we can find some integer $a_I$ with $|\sum_{n \in I} g(n) n^{2 \pi iT} e(a_I n/q)| \ge cH^*$, for some $c >0$. Then $\mathbb{D} (g ; BX^2/H^{2-\rho},B) = O_{\rho,C,c}(1)$ for some $B =O_{\rho,C,c}(1)$.
\end{prop}

Our task is then reduced to establishing the following estimate.

\begin{teo}
\label{3}
Let $g,\rho,\eta$ and $H$ be as in Theorem \ref{1}. Then, if $C$ is sufficiently large with respect to $\eta$ and $\rho$ and $X$ is sufficiently large with respect to $C$, we can find $T \in \R$ and $q \in \N$ with $|T| \le C X^2/H^{2-\rho}$ and $q \le CH^{\rho}$ such that, for  $\ge \frac{X}{H^{1+\rho}}$ disjoint intervals $I_x=[x,x+H^*] \subseteq [X/H^{\rho},2X]$ of length $H^* \in [H^{1-\rho}, H]$ there exists an integer $a_x$ with $|\sum_{n \in I_x} g(n) e(\frac{(n-x) a_x }{q} + \frac{(n-x)T}{x})| \ge  H^*/C$.
\end{teo}

Using Theorem \ref{3}, the Taylor expansion of $(n/x)^{2 \pi i T}$ and the pigeonhole principle, we can then locate an appropriate dyadic interval in $[X/H^{\rho},2X]$ where the hypothesis of Proposition \ref{2} are satisfied, after adjusting $X$, $C$, $H^*$ and $\rho$ if necessary. Thus, in order to prove Theorem \ref{1}, it will suffice to establish Theorem \ref{3}.

\begin{defi}
\label{conf}
Given $R > 1$, $c > 0$ and an interval $I \subseteq \R$, we say a finite set $\mathcal{J} \subseteq I \times \R$ is a $(c,R)$-configuration if $|\mathcal{J}| \ge c|I|/R$ and the first coordinates are $R$-separated points in $I$ (i.e. $|x-y| \ge R$ if $x \neq y$).
\end{defi}

If $g$ satisfies (\ref{A}), then one can use Elliott's inequality and the large sieve to obtain a configuration as in Definition \ref{conf} for which its elements are highly related to each other through a set of primes whose size is a small power of $H$. Concretely, we have the following estimate from \cite{MRT}.

\begin{lema}[\cite{MRT}, Proposition 3.2]
\label{base}
Let the notation and assumptions be as in Theorem \ref{3}. Let $c_0, \varepsilon >0$ be sufficiently small with respect to $\eta$ and $\rho$ and $X$ sufficiently large with respect to $c_0$ and $\varepsilon$. Then, there exists a $(c_0,H/K)$-configuration $\mathcal{J} \subseteq [X/(10K),2X/K] \times \R$, for some $K \in [H^{\varepsilon^2},H^{\varepsilon}]$, such that for every $(x,\alpha) \in \mathcal{J}$ we have
$$ \left| \sum_{x \le n \le x+H/K} g(n) e(\alpha n) \right| \gtrsim H/K,$$
and a pair $P,P' \gtrsim_{\rho,\eta} H^{\varepsilon^2}$ with $P P'=K$ such that, for $\gtrsim_{\rho,\eta} \frac{X}{H} \left( \frac{P}{\log P} \right)^2$ choices of $(x_1,\alpha_1), (x_2,\alpha_2) \in \mathcal{J}$ and $p,q$ primes in $[P,2P]$, we have $|x_1/p - x_2/q| \lesssim_{\rho,\eta} H/(PK)$ and $\| p \alpha_1 - q \alpha_2 \|_{p'} \lesssim_{\rho,\eta} PK/H$ for $\gtrsim_{\rho,\eta} \left( \frac{P'}{\log P'} \right)$ primes $p' \in [P',2P']$. Furthermore, there exist disjoint sets $\mathcal{P}_1, \mathcal{P}_2 \subseteq [P,2P]$ of size $\gtrsim_{\rho,\eta} \frac{P}{\log P}$ such that the same claim holds (up to the implicit constants) if we additionally require $(p,q) \in \mathcal{P}_1 \times \mathcal{P}_2$.
\end{lema}

\begin{proof}
The proof proceeds exactly as in Proposition 3.2 of \cite{MRT}, except for the last claim which simply follows from the pigeonhole principle.
\end{proof}

\begin{notation}
For the rest of this article we will work with a specific of $\eta, \rho, X$ and $H$ as in Theorem \ref{3} and of $c_0, \varepsilon, P, P', K, \mathcal{J}, \mathcal{P}_1$ and $\mathcal{P}_2$ as provided in Lemma \ref{base}. All implicit constants will be allowed to depend on $\eta$ and $\rho$, but will be uniform in our choice of $X$ and $H$.
\end{notation}

We will require the following definition in order to study the properties of pre-paths arising from the configuration $\mathcal{J}$.

\begin{defi}
Let $Q \in \N$ and $\mathcal{J}' \subseteq \mathcal{J}$. We define a \emph{path mod} $Q$ \emph{of length} $k$ \emph{in} $\mathcal{J}'$ to be a pre-path mod $Q$ of length $k$ with parameters $\epsilon_1,\epsilon_1',\ldots,\epsilon_k,\epsilon_k' \lesssim PK/H$, $p_1,\ldots,p_k, q_1,\ldots,q_{k} \in [P,2P]$, $(\alpha_i^{(1)})_{i=1}^{k+1} \subseteq \R / Q \Z$ and a set of elements $(x_1,\alpha_1),\ldots, (x_{k+1},\alpha_{k+1}) \in \mathcal{J}'$ such that $\alpha_i \equiv \alpha_i^{(1)}$ (mod $Q$) for every $1 \le i \le k+1$ and $x_i \frac{q_i}{p_i} = x_{i+1}+O(H/K)$ for every $1 \le i \le k$. We then say $(x_1,\alpha_1)$ and $(x_{k+1},\alpha_{k+1})$ are \emph{connected} by a path mod $Q$ of length $k$ in $\mathcal{J}'$. We call $(x_1,\alpha_1)$ the \emph{initial point} and $(x_{k+1},\alpha_{k+1})$ the \emph{end point} of the path. We additionally say the path is \emph{split} if $p_1,\ldots,p_k \in \mathcal{P}_1$ and $q_1,\ldots,q_{k} \in \mathcal{P}_2$.
\end{defi}

\begin{notation}
Given a path $\ell$ mod $Q$ of length $k$ consisting of elements $(x_1,\alpha_1), \ldots, (x_{k+1}, \alpha_{k+1})$ and primes $(p_1,\ldots,p_{k},q_1,\ldots,q_{k})$ we will sometimes need to consider the 'inverted' path $\ell^{-1}$ having initial point $(x_{k+1}, \alpha_{k+1})$ and end point $(x_1,\alpha_1)$, with the corresponding ordered set of primes given by $(q_{k},\ldots,q_1,p_{k},\ldots,p_1)$. Notice that if $\alpha$ is the top element of a pyramid associated to $\ell$, then it is also the top element of a pyramid associated with $\ell^{-1}$. Also, suppose we are given an additional path $\ell'$ mod $Q$ of length $m$ consisting of elements $(y_1,\beta_1),\ldots,(y_{m+1},\beta_{m+1})$. If $(y_1,\beta_1)=(x_{k+1},\alpha_{k+1})$, we may consider the combined path $\ell + \ell'$ of length $k+m$ with initial point $(x_1,\alpha_1)$ and endpoint $(y_{m+1},\beta_{m+1})$. Notice that if $(\alpha_1^{(j)})_{1 \le j \le k+1}$ is a pyramid associated to $\ell$, then $\ell + \ell'$ will admit a pyramid $(\beta_1^{(j)})_{j=1}^{k+m+1}$ with $\alpha_1^{(j)}=\beta_1^{(j)}$ for every $1 \le j \le k+1$. Finally, we observe that $(\ell + \ell')^{-1} = (\ell')^{-1} + \ell^{-1}$.
\end{notation}

\begin{defi}
Let $Q \in \N$ and $c > 0$. We say $\mathcal{J}' \subseteq \mathcal{J}$ is a $(c,Q)$-\emph{regular} subset of $\mathcal{J}$ if $|\mathcal{J}'| \ge c |\mathcal{J}|$ and every $(x,\alpha) \in \mathcal{J}'$ is connected to $\ge c \frac{P^2}{(\log P)^2}$ other elements of $\mathcal{J}'$ by a split path mod $Q$ of length $1$. 
\end{defi}

We finish this section with the following combinatorial estimate which plays the role of the Blakley-Roy inequality in \cite{MRT, MRTTZ} and turns out to be rather advantageous in practice.

\begin{lema}
\label{regular}
There exists $c \sim 1$ such that, for $\gtrsim \frac{P'}{\log P'}$ primes $p' \in [P',2P']$, we can find a $(c,p')$-regular subset $\mathcal{J}_{p'}$ of $\mathcal{J}$.
\end{lema}

\begin{proof}
For each prime $p' \in [P',2P']$ let $A_{p'}$ be the number of quadruples $((x_1,\alpha_1), (x_2,\alpha_2),p,q) \in \mathcal{J} \times \mathcal{J} \times \mathcal{P}_1 \times \mathcal{P}_2$ with $\| p \alpha_1 - q \alpha_2 \|_{p'} \lesssim PK/H$ and $|x_1/p - x_2/q| \lesssim H/(PK)$. By construction of $\mathcal{J}$, we know that 
$$\sum_{p' \in [P',2P']} A_{p'} \gtrsim \frac{X}{H} \left( \frac{P}{\log P} \right)^2 \frac{P'}{\log P'}.$$
By the prime number theorem, this means that we can find $\gtrsim \frac{P'}{\log P'}$ primes $p' \in [P',2P']$ with $A_{p'} \ge c_1 \frac{X}{H} \left( \frac{P}{\log P} \right)^2$, for some $c_1 \gtrsim 1$. In particular, we have $\ge c_1 \frac{X}{H} \left( \frac{P}{\log P} \right)^2$ paths mod $p'$ of length $1$ in $\mathcal{J}$.

Fix now such a choice of $p'$. Let $\delta > 0$ be a sufficiently small constant and let $\mathcal{J}_1$ be the set of elements of $\mathcal{J}$ that are connected by a path mod $p'$ of length $1$ to at most $\delta  \frac{P^2}{(\log P)^2}$ elements of $\mathcal{J}$. Recursively, let $\mathcal{J}_k$ be the set of elements of $\mathcal{J} \setminus \bigcup_{j=1}^{k-1} \mathcal{J}_{j}$ that are connected by a path mod $p'$ of length $1$ to at most $\delta   \frac{P^2}{(\log P)^2}$ elements of $\mathcal{J} \setminus \bigcup_{j=1}^{k-1} \mathcal{J}_j$. Let $k_0$ be the largest integer such that $\mathcal{J}_{k_0}$ is nonempty (which exists, since $\mathcal{J}$ is finite). Clearly, this means that every element of $\mathcal{J}_{p'} :=\mathcal{J} \setminus \bigcup_{j=1}^{k_0} \mathcal{J}_j$ is connected by a path mod $p'$ of length $1$ to at least $\delta   \frac{P^2}{(\log P)^2}$ elements of $\mathcal{J}_{p'}$. Furthermore, since given $(x_1,\alpha_1), (x_2,\alpha_2) \in \mathcal{J}$ there is at most one choice of $(p,q) \in [P,2P]$ with $|x_1/p - x_2/q| \lesssim H/(PK)$, the number of paths mod $p'$ of length $1$ in $\mathcal{J}_{p'}$ is at least
$$ c_1 \frac{P^2}{(\log P)^2} \frac{X}{H} - 2 \delta \frac{P^2}{(\log P)^2} \sum_{j=1}^{k_0} |\mathcal{J}_j|  \ge (c_1-4\delta) \frac{P^2}{(\log P)^2} \frac{X}{H}.$$
Since every element of $\mathcal{J}_{p'}$ belongs to at most $\lesssim \frac{P^2}{(\log P)^2}$ paths mod $p'$ of length $1$, if $\delta$ is sufficiently small it must be $|\mathcal{J}_{p'}| \gtrsim X/H$. The result follows.
\end{proof}

\section{Uniformity}

In this section we will boost the results of Section \ref{phase} by exploiting the additional information that paths have on the relative sizes of the primes involved. We begin with the following observation in this direction.

\begin{lema}
\label{uni}
There exists $c \gtrsim 1$ such that, given a path mod $Q$ in $\mathcal{J}$ of length $k \le c \log (X/H)$ consisting of primes $p_1, \ldots, p_k$,  $q_1, \ldots, q_k$, we have $\frac{\prod_{i=1}^m q_i}{\prod_{i=1}^m p_i} \sim 1$ for every $1 \le m \le k$.
\end{lema}

\begin{proof}
Fix the path and a choice of $1 \le m \le k$. Since for each $1 \le i \le m$ we have that there is some pair $x_i,x_{i+1}$ in $[X/(10K),2X/K]$ with $(q_i/p_i) x_i = x_{i+1} + O(H/K)$, we have that
$$ \frac{q_i}{p_i} = \frac{x_{i+1}}{x_i} + O\left(\frac{H}{X} \right).$$
In particular, 
$$ \prod_{i=1}^m \frac{q_i}{p_i} = \prod_{i=1}^m \left( \frac{x_{i+1}}{x_i} + O \left(\frac{H}{X} \right) \right).$$
We expand the product in the right into $2^m$ terms, the first of which is $x_{m+1}/x_1 \sim 1$. Since for every $S \subseteq \left\{ 1, \ldots, m \right\}$ it is clearly $\prod_{i \in S} \frac{x_{i+1}}{x_i} \lesssim B^k$ with $B \sim 1$, if we choose $c \gtrsim 1$ sufficiently small in the statement of the current lemma we see that the contribution of the remaining terms is at most 
$$ \lesssim (2B)^{k} \frac{H}{X} \lesssim 1.$$
This proves that $\frac{\prod_{i=1}^m q_i}{\prod_{i=1}^m p_i} \lesssim 1$, while the reverse inequality follows from considering the 'inverted' path and applying the same reasoning.
\end{proof}

\begin{coro}
\label{pd21}
Let the hypothesis be as in Lemma \ref{uni}. Let $1 \le i < j \le k+1$. If $(x_i,\alpha_i), (x_j,\alpha_j)$ are $i$th and $j$th elements of the path, then 
$$|x_i \prod_{t=i}^{j-1} q_t/p_t - x_j| = O((j-i)H/K).$$
\end{coro}

\begin{proof}
This follows immediately upon iterating the relation $x_{t+1} = x_t q_t/p_t + O(H/K)$ and using Lemma \ref{uni}.
\end{proof}

We now insert this information into our previous estimate on pre-paths.

\begin{lema}
\label{alf}
Let the hypothesis be as in Lemma \ref{uni}. Then, for every pyramid $(\alpha_1^{(j)})_{1 \le j \le k+1}$ associated to the path and every $2 \le j \le k+1$ and $1 \le m < j$, we have
$$ \| \left( \prod_{i=m}^{j-1} q_i  \right) \alpha_1^{(j)} - \alpha_1^{(m)} \|_Q \lesssim k \frac{K}{H} \frac{1}{\prod_{i=1}^{m-1} q_i}.$$
\end{lema}

\begin{proof}
By considering shorter paths inside the original path, it will suffice to cover the case in which $\alpha_1^{(j)}$ is the top element of a pyramid associated to a path of length $j-1$. For every $m \le t \le j-1$ we have that $q_t \alpha_1^{(t+1)} \equiv \alpha_1^{(t)} + c_t \, \, (\text{mod }Q)$, where by Lemma \ref{premanli} and Lemma \ref{uni} we know that $c_t$ can be taken to be a real number satisfying
$$ | c_t \left( \prod_{i=1}^{t-1} q_i \right) | \lesssim \frac{P K}{H}  \frac{1 }{p_{ \lfloor (t-1)/2 \rfloor+1}} \prod_{i=1}^{ \lfloor (t-1)/2 \rfloor} q_i/p_i \lesssim \frac{K}{H}.$$
In particular,
$$ | c_t \left( \prod_{i=m}^{t-1} q_i \right) | \lesssim  \frac{K}{H} \frac{1}{\prod_{i=1}^{m-1} q_i}.$$
The result then follows evaluating $\left( \prod_{i=m}^{j-1} q_i \right) \alpha_1^{(j)} = \left( \prod_{i=m}^{j-2} q_i \right) (\alpha_1^{(j-1)} + c_{j-1})$ recursively.
\end{proof}

This, in turn, leads to the following bound.

\begin{lema}
\label{morealf}
Let the hypothesis be as in Lemma \ref{uni}. Then
$$ \| \left( \prod_{i=1}^{j-1} p_i \prod_{i=j}^{k} q_i \right) \alpha_1^{(k+1)} - \alpha_j^{(1)} \|_Q \lesssim k \frac{K}{H},$$
for every $1 \le j \le k+1$.
\end{lema}

\begin{proof}
By Lemma \ref{alf} we know that 
$$ \| \left( \prod_{i=j}^{k} q_i \right) \alpha_1^{(k+1)} - \alpha_1^{(j)} \|_Q \lesssim k \frac{K}{H} \frac{1}{\prod_{i=1}^{j-1} q_i}.$$
Considering the 'inverted' path that goes from $ \alpha_j^{(1)}$ to $\alpha_1^{(1)}$ and applying Lemma \ref{alf} again, we also see that
$$ \| \left( \prod_{i=1}^{j-1} p_i \right) \alpha_1^{(j)} - \alpha_j^{(1)} \|_Q \lesssim k \frac{K}{H}.$$
The result then follows from the triangle inequality and Lemma \ref{uni}.
\end{proof}

\section{Connectedness}

The purpose of this section is to show that a significant proportion of the elements of $\mathcal{J}$ can be connected with each other through a path of reasonable length. Our first observation is that once the initial point of a path of short length is fixed, there are only a few ways of reaching the same end point.

 \begin{lema}
\label{olden}
Let $1 \le k \le \frac{\log (X/(HB \log X))}{2 \log (2P)}$ for some sufficiently large constant $B \sim 1$. Then, for each pair $(x,\alpha), (y,\beta) \in \mathcal{J}$, there are at most $(2k)!$ paths mod $1$ of length $k$ that connect them.
 \end{lema}
 
 \begin{proof}
It will suffice to show that if $(x,\alpha), (y, \beta) \in \mathcal{J}$ are connected by two paths consisting of primes $\left\{ p_1,\ldots,p_{k},q_1,\ldots,q_{k} \right\}$ and $\left\{ p_1',\ldots,p_{k}',q_1',\ldots,q_{k}' \right\}$, respectively, then these two sets of primes must coincide. Let us proceed by contradiction. By Corollary \ref{pd21} we know that 
$$ \left| \left( \prod_{i=1}^{k} \frac{q_i}{p_i} \right) x - y \right| \lesssim \frac{k H}{K}.$$
Similarly, considering the other path and using Lemma \ref{uni} we see that we must also have
$$ \left| \left( \prod_{i=1}^{k} \frac{p_i'}{q'_i} \right) y - x \right| \lesssim  \frac{k H}{K}.$$
Combining both estimates using Lemma \ref{uni} again and the triangle inequality, we obtain
\begin{equation}
\label{old28}
 \left| \prod_{i=1}^{k} \frac{q_i}{p_i} \prod_{i=1}^{k} \frac{p_i'}{q'_i}  - 1\right| \lesssim \frac{kH}{X}.
 \end{equation}
Since the primes in each path are all distinct by definition and the paths do not consist of exactly the same primes, the left-hand side cannot be $0$. But this means it must have size at least $(2P)^{-2k}$. The result then follows from our size hypothesis on $k$.
 \end{proof}
 
 We will be needing the following bound from \cite{MRTTZ}.
 
 \begin{lema}[\cite{MRTTZ}, Lemma 6.1]
 \label{products28}
Let $r \in \N$, $A \ge 1$ and $P_0,N \ge 3$. Then, the number of $2r$-tuples of primes $(p_{1,1},\ldots,p_{1,r},p_{2,1},\ldots,p_{2,r}) \subseteq [P_0,2P_0]^{2r}$ satisfying
$$ \left| \prod_{i=1}^r p_{1,i} - \prod_{i=1}^r p_{2,i} \right| \le A \frac{(2P_0)^{r}}{N}$$
is bounded by
$$ \lesssim A r!^2 (2 P_0)^r \left( \frac{(2 P_0)^r}{N} + 1 \right).$$ 
\end{lema}

The use of this lemma is that it will allow us to assume that no pair of paths mod $1$ of reasonable length that connect the same points share any primes in common. Essentially the same argument was already used in \cite{MRT, MRTTZ}. Precisely, we have the following bound on the number of exceptions.

\begin{coro}
\label{ls28}
Let $(x,\alpha) \in \mathcal{J}$ and let $k$ be as in Lemma \ref{uni}. Then, the number of pairs of split paths mod $1$ of length $k$ with initial point $(x,\alpha)$, sharing the same end point and having at least one prime in common is bounded by
$$ \lesssim (2k)!^2 \frac{(2 P)^{2k}}{\log P} \left( \frac{kH}{X}(2 P)^{2k-1} + 1 \right).$$ 
\end{coro}

\begin{proof}
Let $(p_1,\ldots,p_k,q_1,\ldots,q_k)$ and $(p_1',\ldots,p_k',q_1',\ldots,q_k')$ be the sets of distincts primes corresponding to two paths with initial point $(x,\alpha)$ and common end point $(y,\beta)$, for some $(y,\beta) \in \mathcal{J}$. Using Corollary \ref{pd21} as in the proof of Lemma \ref{olden} we obtain (\ref{old28}), which we may rewrite as
$$ \left| \prod_{i=1}^{k} q_i p_i' - \prod_{i=1}^{k} p_i q_i' \right| \lesssim \frac{kH}{X} (2P)^{2k}.$$
Since $\mathcal{P}_1$ and $\mathcal{P}_2$ are disjoint and the paths share at least a prime in common, we can find a prime appearing in both products on the left-hand side. Factoring it out and applying Lemma \ref{products28} with $r=2k-1$ and $N=\frac{X}{kH}$, we see that the ordered tuple of $4k-2$ remaining primes belongs to a set $S$ of size
$$ \lesssim (2k-1)!^2 (2 P)^{2k-1} \left( \frac{kH}{X}(2 P)^{2k-1} + 1 \right).$$ 
The result follows observing that, for each such element of $S$, there are at most $(2k)^2 \frac{P}{\log P}$ choices of ordered sets of primes $(p_1,\ldots,p_k,q_1,\ldots,q_k)$ and $(p_1',\ldots,p_k', q_1',\ldots,q_k')$ that can lead to it in the above manner.
\end{proof}

We now come to the main estimate of this section, which shows that a large number of elements $(y,\beta) \in \mathcal{J}$ can be connected to a fixed element $(x_0,\alpha_0) \in \mathcal{J}$ through two different paths of small length. In the next section, these two paths will be combined to a single path going through $(y,\beta)$ and having $(x_0,\alpha_0)$ as both its initial and end point.
 
 \begin{lema}
 \label{pairs}
Let $k_0 = \lfloor  \frac{\log (X/(HB \log X))}{2 \log (2P)} \rfloor$ for some sufficiently large $B \sim 1$. Fix $\delta > 0$. Assume $\varepsilon > 0$ in Lemma \ref{base} is sufficiently small with respect to $\delta$, $H \ge \exp(C (\log X)^{1/2} (\log \log X)^{1/2})$ with $C$ sufficiently large with respect to $\delta$ and $\varepsilon$ and $X$ is sufficiently large with respect to $\delta, \varepsilon$ and $C$. Then, there exist $\tau \sim 1$, $(x_0,\alpha_0) \in \mathcal{J}$ and $\mathcal{J}_0 \subseteq \mathcal{J}$ with $|\mathcal{J}_0| \gtrsim \frac{X}{H^{1+\delta}}$ such that the following holds. For every $(y,\beta) \in \mathcal{J}_0$ there exists some $Q_y$ which is the product of $\gtrsim \tau^{k_0} \frac{P'}{\log P'}$ different primes in $[P',2P']$ and such that there are two split paths mod $Q_y$ of length $k_0+2$ connecting $(x_0,\alpha_0)$ with $(y,\beta)$. Furthermore, the two paths share no prime in common.
 \end{lema}
 
 \begin{proof}
During the proof we will let $c \sim 1$ denote a constant that may change at each occurrence. By Lemma \ref{regular} and the pigeonhole principle, we may locate some $(x_0,\alpha_0) \in \mathcal{J}$ that belongs to $\mathcal{J}_{p'}$ for $\gtrsim \frac{P'}{\log P'}$ primes $p' \in [P',2P']$. By regularity of $\mathcal{J}_{p'}$ this means that for each such $p'$ there are $\gtrsim c^{k_0} \left( \frac{P}{\log P} \right)^{2k_0+4} $ split paths mod $p'$ of length $k_0+2$ in $\mathcal{J}_{p'}$ with initial point $(x_0,\alpha_0)$.
 
 Since there are at most $\le D^{k_0} \left( \frac{P}{\log P} \right)^{2k_0+4}$ such paths mod $1$ in $\mathcal{J}$ for some $D \sim 1$, we conclude that there are $\gtrsim c^{k_0} \left( \frac{P}{\log P} \right)^{2k_0+4}$ split paths mod $1$ of length $k_0+2$ with initial point $(x_0,\alpha_0)$ that are also split paths mod $p'$ for $\gtrsim  c^{k_0}  \frac{P'}{\log P'}$ choices of $p' \in [P',2P']$ (which depend on each path). In particular, this means they are split paths mod $Q$, where $Q$ is the product of such primes. Write $\mathcal{R}$ for this set of paths.
 
 Write $\mathcal{R}_y$ for the paths in $\mathcal{R}$ that have endpoint $(y,\beta)$. By construction of $\mathcal{R}$ and the Cauchy-Schwarz inequality, we see that there must be $\gtrsim c^{k_0} |\mathcal{R}_y|^2$ pairs of paths in $\mathcal{R}_y$ such that both paths are split paths mod $p'$ for $\gtrsim c^{k_0} \frac{P'}{\log P'}$ primes $p' \in [P',2P']$ in common. In particular, they will both be split paths mod $Q$, with $Q$ the product of these primes. It follows that the number of pairs of paths in $\mathcal{R}$ having the same endpoint and with both being split paths mod $p'$ for $\gtrsim c^{k_0} \frac{P'}{\log P'}$ primes $p' \in [P',2P']$ is, by Cauchy-Schwarz, at least 
$$ \gtrsim c^{k_0} \frac{|\mathcal{R}|^2 H}{X}  \gtrsim c^{k_0} \frac{H}{X} \left( \frac{P}{\log P} \right)^{4k_0+8}.$$
Since by our hypothesis on $\varepsilon, C$ and $X$ we have that
$$ c^{-k_0} (2(k_0+2))!^2 k_0 (\log P)^{4 k_0+7} (\log X) < P^{1/2},$$
say, it also follows from Corollary \ref{ls28} and our choice of $k_0$ that at least half of these pairs of paths, say, do not share any primes from $[P,2P]$ in common. 

Let now $(y,\beta) \in \mathcal{J}$. If $\ell$ is a path mod $1$ in $\mathcal{R}_y$, it can be written as $\ell_0+\ell_1$, where $\ell_0$ is a path mod $1$ of length $k_0$ and $\ell_1$ is a path mod $1$ of length $2$ with end point $(y,\beta)$. There are $\lesssim \left( \frac{P}{\log P} \right)^4$ choices of $\ell_1$ and each such choice fixes the end point of $\ell_0$. On the other hand, by Lemma \ref{olden} we know that there are at most $(2k_0)!$ choices of $\ell_0$ with the same endpoint. We thus conclude that $(y,\beta)$ is the common endpoint of $\lesssim (2k_0)!^2 \left( \frac{P}{\log P} \right)^8$ of the pairs of paths we have constructed and therefore, by our lower bound on $|\mathcal{R}|$ and our choice of $k_0$, there are
$$ \gtrsim \frac{c^{k_0}}{(2k_0)!^2} \frac{|\mathcal{R}|^2 H}{X}\left( \frac{P}{\log P} \right)^{-8} \gtrsim  \frac{c^{k_0}}{(2k_0)!^2} \frac{X}{H (\log X)^2} P^{-4} (\log P)^{-4 k_0},$$
different elements $(y,\beta) \in \mathcal{J}$ that are connected to $(x_0,\alpha_0)$ by a pair of paths of the desired form. The result then follows, since 
$$ c^{-k_0} (2k_0)!^2 P^4 (\log P)^{4 k_0} (\log X)^2 < H^{\delta},$$
under our hypothesis on $\varepsilon, H$ and $X$.
\end{proof}

\section{Creating a global frequency}

In this section we conclude the proof of Theorem \ref{1}. We begin by finding choices of $T_y$ and $q_y$ that work for each $(y,\beta) \in \mathcal{J}_0$ and then deduce that a common choice of $T$ and $q$ works for many of these elements.

\begin{prop}
Let the notation be as in Lemma \ref{pairs}. Then, for every $(y,\beta) \in \mathcal{J}_0$ there exists some $T_y \in \R$ with $|T_y| \lesssim X^2/H^{2-\delta}$ such that
\begin{equation}
\label{ff2}
\alpha_0 \equiv \frac{a_y}{d_y} Q_y + \frac{T_y}{x_0} + O ( H^{\delta-1} ) \, \, (\text{mod }Q_y),
 \end{equation}
and
\begin{equation}
\label{ff28}
 \beta \equiv \frac{b_y}{d_y} Q_y + \frac{T_y}{y} + O(H^{2\delta-1}) \, \, (\text{mod }Q_y),
 \end{equation}
for some integers $a_y, b_y, d_y \lesssim H^{\delta}$.
\end{prop}

\begin{proof}
Fix $(y,\beta) \in \mathcal{J}_0$. Write $\ell_1$ and $\ell_2$ for a pair of paths mod $Q_y$ of length $k:=k_0+2$ of the kind provided by Lemma \ref{pairs}. If we write $(p_1,\ldots,p_k, q_1,\ldots,q_k)$ and $(p_1',\ldots,p_k',q_1',\ldots,q_k')$ for the primes in $[P,2P]$ associated to $\ell_1$ and $\ell_2$ respectively, then $\ell_1+(\ell_2)^{-1}$ will be a path mod $Q_y$ of length $2k$ consisting of (distinct) primes $(p_1,\ldots,p_k,q_k',\ldots,q_1',q_1,\ldots,q_k,p_k',\ldots,p_1')$ having $(x_0,\alpha_0)$ as its initial and end point. If $\alpha_y \in \R / Q_y \Z$ is the top element of a pyramid associated to this path, then by Lemma \ref{morealf} applied with $j=1$ and $j=2k+1$ and the triangle inequality, we see that
$$ \| \left( \prod_{i=1}^k p_i q_i' - \prod_{i=1}^k q_i p_i' \right) \alpha_y \|_{Q_y} \lesssim k \frac{K}{H},$$
and therefore
\begin{align*}
 \alpha_y &\equiv \frac{u_y}{d_y} Q_y + O \left( k \frac{K}{H} \right) \, \, (\text{mod }Q_y) \\
 &\equiv \frac{u_y}{d_y} Q_y + \frac{T_y}{x_0 \prod_{i=1}^k q_i p_i'} \, \, (\text{mod }Q_y),
 \end{align*}
for some $T_y \in \R$ that by our choice of $k$ can be taken to satisfy 
$$|T_y| \lesssim  k \frac{K}{H} x_0 \left( \prod_{i=1}^k q_i p_i' \right) \lesssim k \frac{K}{H} \frac{X}{K} \left(\frac{X}{H} (2P)^4 \right) \lesssim \frac{X^2}{H^{2-\delta}},$$
and for some integers $|u_y| \le d_y$ with $d_y$ dividing $\left| \prod_{i=1}^k p_i q_i' - \prod_{i=1}^k q_i p_i'  \right|$. Using Corollary \ref{pd21} as in the deduction of (\ref{old28}) we see that
$$ \left| \prod_{i=1}^k p_i q_i' - \prod_{i=1}^k q_i p_i'  \right| \lesssim \frac{k H}{X} (2P)^{2k} \lesssim k (2P)^4,$$
by our choice of $k$. In particular, we have $|d_y| \lesssim H^{\delta}$. Applying Lemma \ref{morealf} with $j=1$ again, we deduce that we can find some integer $|a_y| \le d_y$ with
$$ \alpha_0 \equiv \left( \prod_{i=1}^k q_i p_i' \right) \alpha_y + O \left( k \frac{K}{H} \right) \equiv \frac{a_y}{d_y} Q_y + \frac{T_y}{x_0} + O \left( H^{\delta-1} \right) \, \, (\text{mod }Q_y).$$
Similarly, applying Lemma \ref{morealf} with $j=k+1$, we obtain an integer $|b_y| \le d_y$ with
$$ \beta \equiv \frac{b_y}{d_y} Q_y + \frac{T_y}{x_0} \frac{\prod_{i=1}^k p_i p_i' }{ \prod_{i=1}^k q_i p_i'} +  O \left( H^{\delta-1} \right) \, \, (\text{mod }Q_y).$$
Removing the primes $p_1', \ldots, p_k'$ appearing in both numerator and denominator, and using that by Lemma \ref{uni} and Corollary \ref{pd21} we have
$$ \left| \frac{T_y}{x_0} \frac{\prod_{i=1}^k p_i }{ \prod_{i=1}^k q_i} - \frac{T_y}{y} \right| \lesssim \frac{|T_y|}{(X/K)^2} \left|y- \left(\prod_{i=1}^k q_i/p_i \right) x_0 \right| \lesssim H^{2 \delta-1},$$
we obtain the result.
\end{proof}

We can now conclude the proof of Theorem \ref{3}. Let $\delta >0$ be sufficiently small and let the notation be as in Lemma \ref{pairs}. By Cauchy-Schwarz and the pigeonhole principle, we know from Lemma \ref{pairs} that can find some $c \sim 1$ and $(y_0,\beta_0) \in \mathcal{J}_0$ such that 
$$\tilde{Q}_y := \text{gcd}(Q_{y_0},Q_y) \ge (P')^{c^{k_0}  \frac{P'}{\log P'}} \ge \exp(c^{k_0} P') $$
 for $\ge c^{k_0} |J_0| \gtrsim H^{-\delta} |J_0|$ elements $(y,\beta) \in \mathcal{J}_0$. If we define $T := T_{y_0}$, $Q := Q_{y_0}$, $a := a_{y_0}$ and $d=d_{y_0}$, we see from (\ref{ff2}) that for each such $(y,\beta) \in \mathcal{J}_0$ we have
 $$ \alpha_0 \equiv \frac{a_y}{d_y} Q_y + \frac{T_y}{x_0} +  O(H^{\delta-1}) \equiv \frac{a}{d} Q + \frac{T}{x_0} + O(H^{\delta-1}) \, \, (\text{mod }\tilde{Q}_y).$$
 From our bounds on the individual quantities involved, it follows that it must necessarily be 
 $$  \frac{a_y}{d_y} Q_y \equiv \frac{a}{d} Q \, \, (\text{mod }\tilde{Q}_y),$$
 and therefore it must also be $|T - T_y| \lesssim |x_0|/H^{1-\delta}$. We then conclude from (\ref{ff28}) that
 \begin{equation}
 \label{llu2}
  \beta \equiv \frac{b_y}{d_y} Q_y + \frac{T}{y} +  O(H^{2\delta-1}) \, \, (\text{mod }\tilde{Q}_y).
  \end{equation}
 By the pigeonhole principle, it follows that there is some $q \lesssim H^{\delta}$ such that (\ref{llu2}) holds with $d_y = q$ for $\gtrsim H^{-2 \delta} |\mathcal{J}_0| \gtrsim X/H^{1+3 \delta}$ elements $(y,\beta) \in \mathcal{J}$. For each such $(y,\beta)$, we know by construction of $\mathcal{J}$ and the triangle inequality that we can find some interval $[z,z+H^*] \subseteq [y,y+H/K]$ of length $H^* := H^{1-3 \delta}$ with
 $$ \sum_{z \le n < z+H^*} g(n) e(\beta n) \gtrsim H^*.$$
 Since $|T/y-T/z| \lesssim H^{2\delta-1}$, it then follows that
 $$ \sum_{z \le n \le z+H^*} g(n) e(\frac{(n-z)b_y}{q} Q_y + \frac{(n-z)T}{z}) \gtrsim H^*.$$
 Choosing $\delta$ sufficiently small with respect to $\rho$, this concludes the proof of Theorem \ref{3} and therefore of Theorem \ref{1}.

\end{document}